\documentclass{amsart}
\usepackage{amsmath}\allowdisplaybreaks
\usepackage{amssymb}
\usepackage{mathrsfs}
\usepackage{amsmath}
\usepackage{amsfonts}
\usepackage{graphicx}
\usepackage{xypic}
\newtheorem{thm}{Theorem}[section]

\newtheorem{lem}[thm]{Lemma}
\newtheorem{prop}[thm]{Proposition}
\theoremstyle{definition}
\newtheorem{defn}[thm]{Definition}
\newtheorem{rem}[thm]{Remark}
\numberwithin{equation}{section}
\newtheorem{ack}{Acknowledgements}

\begin{document}
\title[Variational formulas of higher order mean curvatures]{Variational formulas of higher order mean curvatures}
\author[L. Xu]{Ling Xu}\author[J. Q. Ge]{Jianquan Ge}
\thanks{The project is partially supported by the NSFC (No.11001016),
the SRFDP, and the Program for Changjiang Scholars and Innovative
Research Team in University.}
\address{School of Mathematical Sciences, Laboratory of Mathematics and Complex Systems, Beijing Normal
University, Beijing 100875, China} \email{xuling6725@163.com}
\email{jqge@bnu.edu.cn}
\thanks{The second author is the corresponding author.}

\date{}
\begin{abstract}
In this paper, we establish the first variational formula and its
Euler-Lagrange equation for the total $2p$-th mean curvature
functional $\mathcal {M}_{2p}$ of a submanifold $M^n$ in a general
Riemannian manifold $N^{n+m}$ for $p=0,1,\cdots,[\frac{n}{2}]$. As
an example, we prove that closed complex submanifolds in complex
projective spaces are critical points of the functional $\mathcal
{M}_{2p}$, called relatively $2p$-minimal submanifolds, for all $p$.
At last, we discuss the relations between relatively $2p$-minimal
submanifolds and austere submanifolds in real space forms, as well
as a special variational problem.
\end{abstract}
\maketitle
\section{introduction}
It is well known that critical points of the volume functional for
isometric immersions are submanifolds with vanishing mean curvature
vector field. For a hypersurface, the mean curvature vector field is
just given by the mean value of the principal curvatures (up to a
direction). The higher order mean curvatures of a hypersurface are
then defined as the (normalized) higher order elementary symmetric
polynomials of the principal curvatures, whose variational
properties were studied by Reilly \cite{Re1} in real space forms and
by Li \cite{Li1} in general Riemannian manifolds. Reilly \cite{Re2}
also introduced the notion of higher order mean curvatures of
compact submanifolds in Euclidean spaces when studying the first
eigenvalue of the Laplacian. Moreover, he derived the first
variational formula of the integral of each even order mean
curvature. Afterwards, two natural generalizations came into
intensive studies.

One natural way to define the higher order mean curvatures of a
submanifold $M^n$ in a general Riemannian manifold $N^{n+m}$ is by
using the curvature operator $R^M$ (or the curvature forms
$\Omega_{ij}^M$) of the submanifold $M$, in which case the $2p$-th
mean curvature and $(2p+1)$-th mean curvature vector field will be
denoted by $K_{2p}^M,~H_{2p+1}^M$. The other way is to use the
relative curvature operator $R^M-R^N$ (or the relative curvature
forms $\Omega_{ij}^M-\Omega_{ij}^N$) of the immersion $f$ and the
corresponding higher order mean curvatures will be denoted by
$K_{2p}^f,~H_{2p+1}^f$. See section \ref{sec-pre} for explicit
definitions. Note that $H_1^M=H_1^f$ is just the mean curvature
vector field, for hypersurfaces $K_{2p}^f,~H_{2p+1}^f$ are just the
usual higher order mean curvatures, and for submanifolds in
Euclidean spaces $K_{2p}^M=K_{2p}^f, ~H_{2p+1}^M=H_{2p+1}^f$ are
just the higher order mean curvatures defined by Reilly. In general,
$K_{2p}^M$ depends only on the metric of the submanifold and thus is
an intrinsic invariant. It is called the $2p$-th Gauss-Bonnet
curvature by Labbi \cite{La} and its integral is called a Killing
invariant by Li \cite{Li2}. Both of Li \cite{Li2} and Labbi
\cite{La} studied the variational problem of these intrinsic
invariants and characterized the critical points by the vanishing of
$H_{2p+1}^M$ which thereby naturally generalize minimal submanifolds
in a general  Riemannian manifold into $2p$-minimal. On the other
hand, $K_{2p}^f$ is not intrinsic in general. Nevertheless, for
submanifolds in real space forms, it can be expressed as a linear
combination of $1,K_{2}^M,\cdots,K_{2p}^M$ and hence is intrinsic in
this case. Among other things, Li \cite{Li1} calculated the first
variational formula of the integral of $K_{2p}^f$ for submanifolds
in real space forms and for hypersurfaces in general Riemannian
manifolds. In analogy, Cao and Li \cite{Cao-L} considered the
variational problem of the integral of some linear combination of
$K_{2p}^f$s for submanifolds in real space forms so as to
characterize the critical points by the vanishing of $H_{2p+1}^f$,
which they also called $2p$-minimal submanifolds. In addition, they
obtained a non-existence result for closed stable $2p$-minimal
submanifolds in spheres that would reduce to a result of Simons
\cite{Simons} when $p=0$ (Some similar results for hypersurfaces have
been recently obtained by \cite{LJ}). In view of these two lines of
developments, we come to consider the variational problem of the
integral of $K_{2p}^f$ for submanifolds in a general Riemannian
manifold.

In this paper, we establish the first variational formula and its
Euler-Lagrange equation for the functional $\mathcal
{M}_{2p}(f):=\int_M K_{2p}^f dV_M$ defined as the total $2p$-th mean
curvature of a submanifold $M^n$ in a general Riemannian manifold
$N^{n+m}$ for $p=0,1,\cdots,[\frac{n}{2}]$. For hypersurface case
this has been done by Li \cite{Li1}. It is noteworthy to mention
that the object in this variational problem is no longer an
intrinsic invariant as in preceding references. As an example, we
prove that closed complex submanifolds in complex projective spaces
are critical for the functional $\mathcal {M}_{2p}$ for all $p$,
which we called relatively $2p$-minimal. At last, we discuss the
relations between $2p$-minimal submanifolds and austere submanifolds
in real space forms, as well as a special variational problem.
\section{Preliminaries}\label{sec-pre}
We begin with the definition of the $2p$-th mean curvature and
$(2p+1)$-th mean curvature vector field. Throughout this paper, we
adopt the notions used in \cite{Ge}.

 Let $M^n$ and $N^{n+m}$ be
Riemannian manifolds of dimension $n$ and $n+m$ respectively, and
$f:M^n\rightarrow N^{n+m}$ be an isometric immersion. Around each
point in $M$, choose a local orthonormal frame $\{e_1,\ldots,
e_{n+m}\}$ of $TN$ such that $\{e_1,\ldots, e_n\}$ are tangent
vectors of $M$ while $\{e_{n+1},\ldots, e_{n+m}\}$ are normal to
$M$. Then we use $\{\theta_A\mid 1\leq A\leq n+m\}$ and
$\{\theta_{AB}\mid 1\leq A,B\leq n+m\}$ to denote the corresponding
dual 1-forms and connection 1-forms respectively. The following
convention for indices will be used throughout this paper:
$$1\leq i,j,k\leq n,\quad n+1\leq\alpha,\beta,\gamma\leq n+m,\quad 1\leq A,B,C\leq n+m.$$
The structure equations of $N$ are given by
\begin{equation*}
\left\{
  \begin{array}{ll}
    d\theta_A=\sum\limits_B\theta_{AB}\wedge\theta_B,\quad \theta_{AB}=-\theta_{BA},\\
    d\theta_{AB}=\sum\limits_C\theta_{AC}\wedge\theta_{CB}-\Omega^N_{AB},
  \end{array}
\right.
\end{equation*}
where the curvature forms
$\Omega^N_{AB}=\frac{1}{2}\sum_{C,D}R_{ABCD}\theta_C\wedge\theta_D$
and $R_{ABAB}$ is the sectional curvature of $N$ at the two plane
$e_A\wedge e_B$. Comparing with the structure equations of $M$
\begin{equation*}
\left\{
  \begin{array}{ll}
    d\theta_i=\sum\limits_j\theta_{ij}\wedge\theta_j,\quad \theta_{ij}=-\theta_{ji},\\
    d\theta_{ij}=\sum\limits_k\theta_{ik}\wedge\theta_{kj}-\Omega^M_{ij},
  \end{array}
\right.
\end{equation*}
we define the relative curvature forms $\Omega_{ij}$ of the
immersion $f$ by using Gauss equation
\begin{equation}\label{def-Omegaij}
\Omega_{ij}:=
\Omega^M_{ij}-\Omega^N_{ij}=\sum_\alpha\theta_{i\alpha}\wedge\theta_{j\alpha}.
\end{equation}

\begin{defn}
For $p=0,1,\cdots,[\frac{n}{2}]$, the $2p$-th (relative) mean
curvature $K^f_{2p}$ and the $(2p+1)$-th (relative) mean curvature
vector field $H^f_{2p+1}$ of $f$ are defined as follows (cf.
\cite{Ge}):
\begin{equation}\label{def-meancurv}
\begin{array}{lll}
K^f_{2p}=\frac{(n-2p)!}{n!}\sum\limits_{I_{2p}}\Omega_{i_1i_2}\wedge\dots\wedge
\Omega_{i_{2p-1}i_{2p}}(e_{i_1},\ldots,e_{i_{2p}}),\\
H^f_{2p+1}=\frac{(n-2p-1)!}{n!}\sum\limits_{\alpha}\sum\limits_{I_{2p+1}}\Omega_{i_1i_2}\wedge\dots\wedge
\Omega_{i_{2p-1}i_{2p}}\wedge\theta_{i_{2p+1}\alpha}(e_{i_1},\ldots,e_{i_{2p+1}})e_\alpha,
\end{array}
\end{equation}
where the index $I_k=(i_1,\ldots,i_k)$ denotes $k$ different
integers in $\{1,\ldots,n\}$ for $k=1,\dots, n$. We also denote
$K^f_0:=1$, $H^f_{-1}:=H^f_{n+1}:=0$.
\end{defn}

One can easily find that $K^f_{2p}$ and $H^f_{2p+1}$ are independent
of the choice of the local frame and hence well-defined (cf.
\cite{Ge}). In analogy, the $2p$-th Gauss-Bonnet curvature
$K^M_{2p}$ and the $(2p+1)$-th mean curvature vector field
$H^M_{2p+1}$ introduced in last section can be defined by the same
formulas of (\ref{def-meancurv}) with $\Omega_{ij}^M$ instead of all
$\Omega_{ij}$ therein. When $N^{n+m}$ is the real space form
$\mathbb{R}^{n+m}(c)$ of constant sectional curvature $c$, a
straightforward calculation shows that the two families can express
each other by
\begin{equation}\label{rel-Hf-HM}
K^M_{2p}=\sum_{k=0}^pc^{p-k}(^p_k)K^f_{2k}, \quad
H^M_{2p+1}=\sum_{k=0}^pc^{p-k}(^p_k)H^f_{2p+1}.
\end{equation}

If $M^n$ is compact, possibly with boundary, the total $2p$-th mean
curvature of $f$ is given by the integral
\begin{equation}\label{def-M2pf}
\mathcal {M}_{2p}(f):=\int_M K_{2p}^f dV_M.
\end{equation}
We apply a variation of the immersion $f$ as follows: Let $I$ be the
interval $-\frac{1}{2}<t<\frac{1}{2}$. Let $F: M\times I\rightarrow
N$ be a differentiable mapping such that its restriction to $M\times
t\:(t\in I)$, is an immersion, denoted by $f_t$, and that
$F(x,0)=f(x)$ for $x\in M$. Our aim is to evaluate the first
variational formula of the functional $\mathcal {M}_{2p}$ under such
variations, that is to calculate
\begin{equation}\label{variation0}
\frac{d}{dt}\mathcal {M}_{2p}(f_t)\Big|_{t=0}.
\end{equation}
To treat with this type of variational problems, we would like to
apply the moving frame method presented by Chern in \cite{Chern}.
Choose a local orthonormal frame field $\{e_A(x,t)\}$ of $TN$ over
$M\times I$ such that for every $t\in I$, $e_i(x,t)$ are tangent
vectors to $M_t:=f_t(M)=F(M\times t)$ at $(x,t)$ and hence
$e_{\alpha}(x,t)$ are normal vectors. Let $\omega_A,~\omega_{AB}$ be
the corresponding dual $1$-forms and connection $1$-forms of $N$
over $M\times I$. Then they can be written as
\begin{equation}\label{def-var-frame}
 \omega_i=\theta_i+a_idt,\quad \omega_\alpha=a_\alpha dt,\quad
 \omega_{AB}=\theta_{AB}+a_{AB}dt,
\end{equation}
where $\theta_i,\theta_{AB}$ are linear differential forms in $M$
with coefficients which may depend on $t$. For $t=0$ they reduce to
the forms with the same notation on $M$. The vector
$\nu:=\sum_Aa_Ae_A(x,0)=\frac{d}{dt}F(x,t)|_{t=0}$ is called the
deformation vector. We write the exterior differential operator $d$
on $M\times I$ as
\begin{equation*}
d=d_M+dt\frac{\partial}{\partial t}.
\end{equation*}
Now by the definition of $2p$-th mean curvature, we have
\begin{eqnarray}\label{trans-K2p}
K^{f_t}_{2p}dV_{M_t}&=&\frac{(n-2p)!}{n!}\sum\limits_{I_{2p}}\Omega_{i_1i_2}\wedge\dots\wedge
\Omega_{i_{2p-1}i_{2p}}(e_{i_1},\ldots,e_{i_{2p}})dV_{M_t}\\
&=&\frac{(n-2p)!}{n!}\sum\limits_{I_{2p}}\Omega_{i_1i_2}\wedge\dots\wedge
\Omega_{i_{2p-1}i_{2p}}(e_{i_1},\ldots,e_{i_{2p}})\theta_{1}\wedge\dots\wedge\theta_{n}\nonumber\\
&=&\frac{1}{n!}\sum\limits_{I_{n}}\delta_{I_n}\Omega_{i_1i_2}\wedge\dots\wedge
\Omega_{i_{2p-1}i_{2p}}\wedge\theta_{i_{2p+1}}\wedge\dots\wedge\theta_{i_n},\nonumber
\end{eqnarray}
where $\delta_{I_n}:=\delta^{1,\ldots,n}_{i_1,\ldots, i_n}$ denotes
the generalized Kronecker symbol. Similarly, we have
\begin{equation}\label{trans-H2p+1}
\langle H^{f_t}_{2p+1},\nu\rangle
dV_{M_t}=\frac{1}{n!}\sum_\alpha\sum\limits_{I_n}a_\alpha
\delta_{I_n}\Omega_{i_1i_2}\wedge\dots\wedge\Omega_{i_{2p-1}i_{2p}}\wedge
\theta_{i_{2p+1}\alpha}\wedge\theta_{i_{2p+2}}\wedge\dots\wedge\theta_{i_n}.
\end{equation} Define an $n$-form on $M$
\begin{equation}\label{def-Theta2p}
\Theta_{2p}=\sum\limits_{I_{n}}\delta_{I_n}\Omega_{i_1i_2}\wedge\dots\wedge
\Omega_{i_{2p-1}i_{2p}}\wedge\theta_{i_{2p+1}}\wedge\dots\wedge\theta_{i_n}.\\
\end{equation}
 Then by (\ref{trans-K2p}) our variational problem (\ref{variation0}) turns to
\begin{equation}\label{variation1}
\frac{d}{dt}\mathcal
{M}_{2p}(f_t)\Big|_{t=0}=\frac{d}{dt}\int_{M_t}K^{f_t}_{2p}dV_{M_t}\Big|_{t=0}
=\frac{1}{n!}\int_{M}\frac{\partial}{\partial
t}\Theta_{2p}\Big|_{t=0}.
\end{equation}

\section{Variational formula of the total $(2p)$-th mean curvature}
In this section we will calculate in detail the first variational
formula of the total $2p$-th mean curvature $\mathcal {M}_{2p}(f)$
in (\ref{def-M2pf}) by moving frame method.

From last section, it suffices to calculate formula
(\ref{variation1}). Recalling the definition of $\Omega_{ij}$ in
(\ref{def-Omegaij}), we put
$\widetilde{\Omega}_{ij}:=\sum_\alpha\omega_{i\alpha}\wedge\omega_{j\alpha}$
where $\omega_{i\alpha}$ is the connection 1-form given in
(\ref{def-var-frame}). Then substituting
$\widetilde{\Omega}_{ij},\omega_i$ for $\Omega_{ij}, \theta_{i}$
into (\ref{def-Theta2p}) respectively, we can define an $n$-form
$\Psi_{2p}$ on $M\times I$:
\begin{equation}\label{def-Psi2p}
\Psi_{2p}=\sum_{I_{n}}\delta_{I_n}\widetilde{\Omega}_{i_1i_2}
\wedge\dots\wedge\widetilde{\Omega}_{i_{2p-1}i_{2p}}\wedge\omega_{i_{2p+1}}\wedge\dots\wedge\omega_{i_n}.
\end{equation}
It is easily seen from (\ref{def-var-frame}, \ref{def-Theta2p}) that
\begin{equation}\label{decom-Psi2p}
\Psi_{2p}=\Theta_{2p}+dt\wedge\Phi_{2p},
\end{equation}
 where
\begin{eqnarray}\label{def-Phi2p}
\Phi_{2p}&=&-2p\sum_{I_n,\alpha}\delta_{I_n}a_{i_{2p}\alpha}
\Omega_{i_1i_2}\wedge\dots\wedge\Omega_{i_{2p-3}i_{2p-2}}\wedge\theta_{i_{2p-1}\alpha}\wedge\theta_{i_{2p+1}}\wedge\dots\wedge\theta_{i_n}\\
&&+(n-2p)\sum_{I_{n}}\delta_{I_n}a_{i_{2p+1}}\Omega_{i_1i_2}
\wedge\dots\wedge\Omega_{i_{2p-1}i_{2p}}\wedge\theta_{i_{2p+2}}\wedge\dots\wedge\theta_{i_n}.\nonumber
\end{eqnarray}
Then taking exterior differential of the equation
(\ref{decom-Psi2p}) we get
\begin{equation}\label{dPsi2p1}
d\Psi_{2p}=d_M\Theta_{2p}+dt\wedge\frac{\partial}{\partial
t}\Theta_{2p}-dt\wedge d_M\Phi_{2p}.
\end{equation}
On the other hand, $d\Psi_{2p}$ can be calculated directly from
(\ref{def-Psi2p}) by using the structure equations of $N$ as the
following.
\begin{lem}\label{Lem-dPsi2p}
 Notations as above, then
\begin{equation}\label{dPsi2p-Lem}
\begin{array}{ccl}
d\Psi_{2p}&=&(n-2p)\sum\limits_\alpha\sum\limits_{I_n}\delta_{I_n}\widetilde{\Omega}_{i_1i_2}\wedge\dots\wedge\widetilde{\Omega}_{i_{2p-1}i_{2p}}\wedge
\omega_{i_{2p+1}\alpha}\wedge\omega_{\alpha}\wedge\omega_{i_{2p+2}}\wedge\dots\wedge\omega_{i_n}\\
&&+2p\sum\limits_\alpha\sum\limits_{I_n}\delta_{I_n}\widetilde{\Omega}_{i_1i_2}\wedge\dots\wedge\widetilde{\Omega}_{i_{2p-3}i_{2p-2}}\wedge
\omega_{i_{2p-1}\alpha}\wedge\Omega^N_{i_{2p}\alpha}\wedge\omega_{i_{2p+1}}\wedge\dots\wedge\omega_{i_n}.
\end{array}
\end{equation}
\end{lem}
\begin{proof}
Using the structure equations of $N$ and interchanging the indices
whenever there occur two essentially equal terms, we can obtain the
following expression:
\begin{eqnarray*}
d\Psi_{2p}&=&-2p\sum\limits_{\alpha}\sum\limits_{I_n}\delta_{I_n}
\widetilde{\Omega}_{i_1i_2}\wedge\dots\wedge\widetilde{\Omega}_{i_{2p-3}i_{2p-2}}
\wedge\omega_{i_{2p-1}\alpha}\wedge d\omega_{i_{2p}\alpha}\wedge
\omega_{i_{2p+1}}\wedge\dots\wedge\omega_{i_n}\\
&&+(n-2p)\sum\limits_{I_n}\delta_{I_n}\widetilde{\Omega}_{i_1i_2}
\wedge\dots\wedge\widetilde{\Omega}_{i_{2p-1}i_{2p}}\wedge
d\omega_{i_{2p+1}}\wedge\omega_{i_{2p+2}}\wedge\dots\wedge\omega_{i_n}\\
&=&-2p\sum\limits_{\alpha}\sum\limits_{I_n,j}\delta_{I_n}\widetilde{\Omega}_{i_1i_2}\wedge\dots\wedge
\widetilde{\Omega}_{i_{2p-3}i_{2p-2}}\wedge\omega_{i_{2p-1}\alpha}\wedge
\omega_{i_{2p}j}\wedge\omega_{j\alpha}\wedge
\omega_{i_{2p+1}}\wedge\dots\wedge\omega_{i_n}\\
&&-2p\sum\limits_{\alpha,\beta}\sum\limits_{I_n}\delta_{I_n}\widetilde{\Omega}_{i_1i_2}\wedge\dots\wedge
\widetilde{\Omega}_{i_{2p-3}i_{2p-2}}\wedge\omega_{i_{2p-1}\alpha}\wedge
\omega_{i_{2p}\beta}\wedge\omega_{\beta\alpha}
\wedge\omega_{i_{2p+1}}\wedge\dots\wedge\omega_{i_n}\\
&&+2p\sum\limits_{\alpha}\sum\limits_{I_n}\delta_{I_n}\widetilde{\Omega}_{i_1i_2}\wedge\dots\wedge
\widetilde{\Omega}_{i_{2p-3}i_{2p-2}}\wedge\omega_{i_{2p-1}\alpha}\wedge
\Omega^N_{i_{2p}\alpha}\wedge
\omega_{i_{2p+1}}\wedge\dots\wedge\omega_{i_n}\\
&&+(n-2p)\sum\limits_{I_n,j}\delta_{I_n}\widetilde{\Omega}_{i_1i_2}\wedge\dots\wedge
\widetilde{\Omega}_{i_{2p-1}i_{2p}}\wedge
\omega_{i_{2p+1}j}\wedge\omega_{j}\wedge\omega_{i_{2p+2}}\wedge\dots\wedge\omega_{i_n}\\
&&+(n-2p)\sum\limits_{\alpha}\sum\limits_{I_n}\delta_{I_n}\widetilde{\Omega}_{i_1i_2}\wedge\dots\wedge\widetilde{\Omega}_{i_{2p-1}i_{2p}}\wedge
\omega_{i_{2p+1}\alpha}\wedge\omega_{\alpha}\wedge\omega_{i_{2p+2}}\wedge\dots\wedge\omega_{i_n}.
\end{eqnarray*}
Since the index $I_n=(i_1,\ldots,i_n)$ is a permutation of
$\{1,\ldots,n\}$, the sum over $j$ from $1$ to $n$ can be looked as
from $i_1$ to $i_n$, which leads to the following:
\begin{eqnarray*}
&&-2p\sum\limits_{\alpha}\sum\limits_{I_n,j}\delta_{I_n}\widetilde{\Omega}_{i_1i_2}\wedge\dots\wedge
\widetilde{\Omega}_{i_{2p-3}i_{2p-2}}\wedge\omega_{i_{2p-1}\alpha}\wedge
\omega_{i_{2p}j}\wedge\omega_{j\alpha}\wedge
\omega_{i_{2p+1}}\wedge\dots\wedge\omega_{i_n}\\
&=&2p\sum_{I_n,j}\delta_{I_n}\widetilde{\Omega}_{i_1i_2}\wedge\dots\wedge
\widetilde{\Omega}_{i_{2p-3}i_{2p-2}}\wedge\widetilde{\Omega}_{i_{2p-1}j}\wedge\omega_{i_{2p}j}\wedge
\omega_{i_{2p+1}}\wedge\dots\wedge\omega_{i_n}\\
&=&2p(2p-2)\sum_{I_n}\delta_{I_n}\widetilde{\Omega}_{i_1i_2}\wedge\dots\wedge
\widetilde{\Omega}_{i_{2p-3}i_{2p-2}}\wedge\widetilde{\Omega}_{i_{2p-1}i_{2p-2}}\wedge\omega_{i_{2p}i_{2p-2}}\wedge
\omega_{i_{2p+1}}\wedge\dots\wedge\omega_{i_n}\\
&&+2p(n-2p)\sum_{I_n}\delta_{I_n}\widetilde{\Omega}_{i_1i_2}\wedge\dots\wedge
\widetilde{\Omega}_{i_{2p-3}i_{2p-2}}\wedge\widetilde{\Omega}_{i_{2p-1}i_{2p+1}}\wedge\omega_{i_{2p}i_{2p+1}}\wedge
\omega_{i_{2p+1}}\wedge\dots\wedge\omega_{i_n}\\
&=&0+2p(n-2p)\sum_{I_n}\delta_{I_n}\widetilde{\Omega}_{i_1i_2}\wedge\dots\wedge
\widetilde{\Omega}_{i_{2p-3}i_{2p-2}}\wedge\widetilde{\Omega}_{i_{2p-1}i_{2p+1}}\wedge\omega_{i_{2p}i_{2p+1}}\wedge
\omega_{i_{2p+1}}\wedge\dots\wedge\omega_{i_n}\\
&=&-2p(n-2p)\sum\limits_{I_n}\delta_{I_n}\widetilde{\Omega}_{i_1i_2}\wedge\dots\wedge\widetilde{\Omega}_{i_{2p-1}i_{2p}}\wedge
\omega_{i_{2p+1}i_{2p}}\wedge\omega_{i_{2p}}\wedge\omega_{i_{2p+2}}\wedge\dots\wedge\omega_{i_n}\quad
(i_{2p}\leftrightarrow i_{2p+1}),
\end{eqnarray*}
where the vanishing of the third line can be easily checked by
exchanging the indices $i_{2p-1},i_{2p-3}$. Similarly,
\begin{eqnarray*}
&&(n-2p)\sum\limits_{I_n,j}\delta_{I_n}\widetilde{\Omega}_{i_1i_2}\wedge\dots\wedge
\widetilde{\Omega}_{i_{2p-1}i_{2p}}\wedge
\omega_{i_{2p+1}j}\wedge\omega_{j}\wedge\omega_{i_{2p+2}}\wedge\dots\wedge\omega_{i_n}\\
&=&2p(n-2p)\sum\limits_{I_n}\delta_{I_n}\widetilde{\Omega}_{i_1i_2}\wedge\dots\wedge\widetilde{\Omega}_{i_{2p-1}i_{2p}}\wedge
\omega_{i_{2p+1}i_{2p}}\wedge\omega_{i_{2p}}\wedge\omega_{i_{2p+2}}\wedge\dots\wedge\omega_{i_n}.
\end{eqnarray*}
Combining with
\begin{eqnarray*}
&&\sum\limits_{\alpha,\beta}\sum\limits_{I_n}\delta_{I_n}\widetilde{\Omega}_{i_1i_2}
\wedge\dots\wedge\widetilde{\Omega}_{i_{2p-3}i_{2p-2}}\wedge\omega_{i_{2p-1}\alpha}
\wedge \omega_{i_{2p}\beta}\wedge\omega_{\beta\alpha}\wedge\omega_{i_{2p+1}}\wedge\dots\wedge\omega_{i_n}\\
&=&-\sum\limits_{\alpha,\beta}\sum\limits_{I_n}\delta_{I_n}\widetilde{\Omega}_{i_1i_2}
\wedge\dots\wedge\widetilde{\Omega}_{i_{2p-3}i_{2p-2}}\wedge
\omega_{i_{2p}\beta}\wedge
\omega_{i_{2p-1}\alpha}\wedge\omega_{\alpha\beta}\wedge\omega_{i_{2p+1}}\wedge\dots\wedge\omega_{i_n}
\quad (^{i_{2p-1}\leftrightarrow i_{2p}}_{\,\,\quad\alpha\leftrightarrow\beta})\\
&=&-\sum\limits_{\alpha,\beta}\sum\limits_{I_n}\delta_{I_n}\widetilde{\Omega}_{i_1i_2}
\wedge\dots\wedge\widetilde{\Omega}_{i_{2p-3}i_{2p-2}}\wedge\omega_{i_{2p-1}\alpha}
\wedge \omega_{i_{2p}\beta}\wedge\omega_{\beta\alpha}\wedge\omega_{i_{2p+1}}\wedge\dots\wedge\omega_{i_n}\\
&=&0,
\end{eqnarray*}
we complete the proof of Lemma \ref{Lem-dPsi2p}.
\end{proof}

From Lemma \ref{Lem-dPsi2p} we can divide the expansion of
$d\Psi_{2p}$ into two parts: one part involving $dt$ and the other
not. In what follows the part of $dt$ in ($\ref{dPsi2p-Lem}$) will
be calculated, since we want to get the expression of
$\frac{\partial}{\partial t}\Theta_{2p}$ concretely by comparing
with the corresponding terms in (\ref{dPsi2p1}).

Substituting into the first term of $d\Psi_{2p}$ in
(\ref{dPsi2p-Lem}) the expression of $\omega_{A},\omega_{AB}$ in
(\ref{def-var-frame}) and recalling (\ref{trans-H2p+1}), we get
\begin{eqnarray}\label{dPsi2p-1st term}
&&(n-2p)\sum\limits_\alpha\sum\limits_{I_n}\delta_{I_n}\widetilde{\Omega}_{i_1i_2}\wedge\dots\wedge\widetilde{\Omega}_{i_{2p-1}i_{2p}}\wedge
\omega_{i_{2p+1}\alpha}\wedge\omega_{\alpha}\wedge\omega_{i_{2p+2}}\wedge\dots\wedge\omega_{i_n}\\
&=&-(n-2p)\sum\limits_\alpha\sum\limits_{I_n}a_\alpha
dt\wedge\delta_{I_n}\Omega_{i_1i_2}\wedge\dots\wedge\Omega_{i_{2p-1}i_{2p}}\wedge
\theta_{i_{2p+1}\alpha}\wedge\theta_{i_{2p+2}}\wedge\dots\wedge\theta_{i_n}\nonumber\\
&=&-(n-2p)n!dt\wedge\langle H^{f_t}_{2p+1},\nu\rangle
dV_{M_t}.\nonumber
\end{eqnarray}
Recall that $\Omega^N_{i\alpha}=\frac{1}{2}\sum_{C,D}R_{i\alpha
CD}\omega_C\wedge\omega_D$. The second term of (\ref{dPsi2p-Lem})
turns to
\begin{eqnarray}\label{dPsi2p-2nd term}
&&2p\sum\limits_\alpha\sum\limits_{I_n}\delta_{I_n}\widetilde{\Omega}_{i_1i_2}\wedge\dots\wedge
\widetilde{\Omega}_{i_{2p-3}i_{2p-2}}\wedge\omega_{i_{2p-1}\alpha}\wedge\Omega^N_{i_{2p}\alpha}
\wedge\omega_{i_{2p+1}}\wedge\dots\wedge\omega_{i_n}\\
&=&2p\sum\limits_{\alpha,\beta}\sum\limits_{I_n,j}\delta_{I_n}R_{i_{2p}\alpha
j\beta}
\widetilde{\Omega}_{i_1i_2}\wedge\dots\wedge\widetilde{\Omega}_{i_{2p-3}i_{2p-2}}\wedge
\omega_{i_{2p-1}\alpha}\wedge\omega_{j}\wedge\omega_{\beta}\wedge\omega_{i_{2p+1}}\wedge\dots\wedge\omega_{i_n}\nonumber\\
&&+p\sum\limits_{\alpha}\sum\limits_{I_n,j,k}\delta_{I_n}R_{i_{2p}\alpha
jk}\widetilde{\Omega}_{i_1i_2}
\wedge\dots\wedge\widetilde{\Omega}_{i_{2p-3}i_{2p-2}}\wedge
\omega_{i_{2p-1}\alpha}\wedge\omega_{j}
\wedge\omega_{k}\wedge\omega_{i_{2p+1}}\wedge\dots\wedge\omega_{i_n}\nonumber\\
&=:&\Gamma_1+\Gamma_2.\nonumber
\end{eqnarray}
To simplify the notation, we put
\begin{equation}\label{def-Omega_I2p}
\Omega_{I_{2p}}:=\Omega_{i_1i_2}\wedge\dots\wedge\Omega_{i_{2p-1}i_{2p}}.
\end{equation}
Then we can get the expression of $\Gamma_1$ as the following:
\begin{eqnarray*}
\Gamma_1&=&2p\sum\limits_{\alpha,\beta}\sum\limits_{I_n,j}
\delta_{I_n}a_{\alpha}dt\wedge R_{i_{2p}\beta
j\alpha}\Omega_{I_{2p-2}} \wedge
\theta_{i_{2p-1}\beta}\wedge\theta_{j}\wedge\theta_{i_{2p+1}}
\wedge\dots\wedge\theta_{i_n}\\
&=&2p(n-2p)!\sum\limits_{\alpha,\beta}\sum\limits_{I_{2p},j}a_\alpha
dt\wedge R_{i_{2p}\beta j\alpha}\Omega_{I_{2p-2}}\wedge\theta_{i_{2p-1}
\beta}\wedge \theta_j(e_{i_1},\ldots,e_{i_{2p}})dV_{M_t}\\
&=&\frac{2p(n-2p)!}{(2p-1)!}\sum\limits_{\alpha,\beta}\sum_{I_{2p-1},i}
\sum_{J_{2p-1},j}a_\alpha dt\wedge \delta ^{i_1,\ldots,i_{2p-1},i}_{j_1,\ldots,j_{2p-1},j}
R_{{i}\beta j\alpha}\Omega_{I_{2p-2}}\wedge\theta_{i_{2p-1}\beta}(e_{j_1},\ldots,e_{j_{2p-1}})dV_{M_t}.\\
\end{eqnarray*}
Similarly, we can compute the $dt$ part of $\Gamma_2$ (denoted by
$\widetilde{\Gamma}_2$) as well:
\begin{eqnarray*}
\widetilde{\Gamma}_2&=&-p(n-2p)\sum\limits_{\alpha}\sum\limits_{I_n,j,k}\delta_{I_n}a_{i_{2p+1}}dt\wedge
R_{i_{2p}\alpha jk}\Omega_{I_{2p-2}}
\wedge\theta_{i_{2p-1}\alpha}\wedge\theta_j\wedge\theta_k\wedge\theta_{i_{2p+2}}\wedge\dots\wedge\theta_{i_n}\\
&&-2p\sum\limits_{\alpha}\sum\limits_{I_n,j,k}\delta_{I_n}a_{j}dt\wedge
R_{i_{2p}\alpha jk}\Omega_{I_{2p-2}}
\wedge\theta_{i_{2p-1}\alpha}\wedge\theta_k\wedge\theta_{i_{2p+1}}\wedge\dots\wedge\theta_{i_n}\\
&&+p\sum\limits_{\alpha}\sum\limits_{I_n,j,k}\delta_{I_n}a_{i_{2p-1}\alpha}dt\wedge
R_{i_{2p}\alpha jk}
\Omega_{I_{2p-2}}\wedge \theta_{j}\wedge\theta_{k}\wedge\theta_{i_{2p+1}}\wedge\dots\wedge\theta_{i_n}\\
&&-2p(p-1)\sum\limits_{\alpha,\beta}\sum\limits_{I_n,j,k}\delta_{I_n}a_{i_{2p-2}\beta}dt\wedge
R_{i_{2p}\alpha jk}
\Omega_{I_{2p-4}}\wedge\theta_{i_{2p-3}\beta}\wedge\theta_{i_{2p-1}\alpha}\wedge \theta_{j}\wedge\theta_{k}\wedge\theta_{i_{2p+1}}\wedge\dots\wedge\theta_{i_n}\\
&=&\frac{p(n-2p)!}{(2p-1)!}\sum\limits_{\alpha}\sum\limits_{I_{2p-1},i,i'}\sum_{J_{2p-1},j,j'}a_{i}dt\wedge
\delta^{i_1,\ldots,i_{2p-1},i,i'}_{j_1,\ldots,j_{2p-1},j,j'}R_{i'\alpha
jj'}\Omega_{I_{2p-2}}
\wedge\theta_{i_{2p-1}\alpha}(e_{j_1},\ldots,e_{j_{2p-1}})dV_{M_t}\\
&&-\frac{2p(n-2p)!}{(2p-1)!}\sum\limits_{\alpha}\sum_{I_{2p-1},i,i'}\sum_{J_{2p-1},j'}a_{i}dt\wedge
\delta ^{i_1,\ldots,i_{2p-1},i'}_{j_1,\ldots,j_{2p-1},j'}R_{i'\alpha
ij'}\Omega_{I_{2p-2}}\wedge\theta_{i_{2p-1}\alpha}(e_{j_1},\ldots,e_{j_{2p-1}})dV_{M_t}\\
&&+\frac{p(n-2p)!}{(2p-2)!}\sum_{\alpha}\sum_{I_{2p-2},i,i'}\sum_{J_{2p-2},j,j'}\Big(a_{i\alpha}dt
\wedge\delta
^{i_1,\ldots,i_{2p-2},i,i'}_{j_1,\ldots,j_{2p-2},j,j'}R_{i'\alpha
jj'}
\Omega_{I_{2p-2}}(e_{j_1},\dots,e_{j_{2p-2}})dV_{M_t}\\
&&+2(p-1)\sum_{\beta}a_{i\alpha} dt\wedge \delta
^{i_1,\ldots,i_{2p-2},i,i'}_{j_1,\ldots,j_{2p-2},j,j'}R_{i'\beta
jj'}\Omega_{I_{2p-4}}\wedge\theta_{i_{2p-3}\alpha}\wedge\theta_{i_{2p-2}\beta}
(e_{j_1},\ldots,e_{j_{2p-2}})dV_{M_t}\Big).\\
\end{eqnarray*}
Now we are ready to give the first variational formula.
\begin{thm}\label{Thm-varformula}
 Let $f:M^{n}\rightarrow N^{n+m}$ be an isometric immersion from a compact manifold $M$ (possibly with boundary)
into a Riemannian manifold $N$. Then for
$p=0,1,\cdots,[\frac{n}{2}]$, the first variational formula of the
total $2p$-th mean curvature $\mathcal{M}_{2p}(f)$ in
(\ref{def-M2pf}) is given by
\begin{equation*}
\frac{d}{dt}\mathcal{M}_{2p}(f_t)\Big|_{t=0}=\int_M\Big(\langle
-(n-2p)H^f_{2p+1}+pW_{2p-1},\nu\rangle+p\sum_i\langle
Q^i_{2p-2},\nabla_{e_i}\nu\rangle\Big)dV_M+\frac{1}{n!}\int_{\partial
M}\Phi_{2p}.
\end{equation*}
Here $\nu$ is the deformation vector field, $\nabla$ is the
Levi-Civita connection of $N$, $H^f_{2p+1}$ is the $(2p+1)$-th mean
curvature vector field, $\Phi_{2p}$ is defined in (\ref{def-Phi2p}),
$W_{2p-1}$ and $Q^i_{2p-2}$ are defined by
\begin{eqnarray*}
W_{2p-1}&=&\frac{(n-2p)!}{(2p-1)!n!}\sum_{\alpha,\beta}\sum_{I_{2p-1},i}\sum_{J_{2p-1},j}
\delta ^{i_1,\ldots,i_{2p-1},i}_{j_1,\ldots,j_{2p-1},j}R_{{i}\beta j\alpha}
\Omega_{I_{2p-2}}\wedge\theta_{i_{2p-1}\beta}(e_{j_1},\ldots,e_{j_{2p-1}})e_\alpha\\
&&+\frac{(n-2p)!}{(2p-1)!n!}\sum_{\alpha}\sum_{I_{2p-1},i,i'}\sum_{J_{2p-1},j'}\Big(\sum_j\delta
^{i_1,\ldots,i_{2p-1},i,i'}_{j_1,\ldots,j_{2p-1},j,j'}R_{i'\alpha
jj'}\Omega_{I_{2p-2}}\wedge\theta_{i_{2p-1}\alpha}(e_{j_1},\ldots,e_{j_{2p-1}})\\
&&-2\delta
^{i_1,\ldots,i_{2p-1},i'}_{j_1,\ldots,j_{2p-1},j'}R_{i'\alpha ij'}
\Omega_{I_{2p-2}}\wedge\theta_{i_{2p-1}\alpha}(e_{j_1},\ldots,e_{j_{2p-1}})\Big)e_i,\\
Q^i_{2p-2}&=&\frac{(n-2p)!}{(2p-2)!n!}\sum\limits_\alpha\sum\limits_{I_{2p-2},i'}
\sum_{J_{2p-2},j,j'}\Big(\delta^{i_1,\ldots,i_{2p-2},i,i'}_{j_1,\ldots,j_{2p-2},j,j'}
R_{i'\alpha jj'}\Omega_{I_{2p-2}}(e_{j_1},\dots,e_{j_{2p-2}})\\
&&+2(p-1)\sum_{\beta}\delta^{i_1,\ldots,i_{2p-2},i,i'}_{j_1,\ldots,j_{2p-2},j,j'}R_{i'\beta
jj'}\Omega_{I_{2p-4}}\wedge\theta_{i_{2p-3}\alpha}\wedge\theta_{i_{2p-2}\beta}
(e_{j_1},\ldots,e_{j_{2p-2}})\Big)e_\alpha,
\end{eqnarray*}
where
$\Omega_{I_{2p}}=\Omega_{i_1i_2}\wedge\dots\wedge\Omega_{i_{2p-1}i_{2p}}$
is defined in (\ref{def-Omega_I2p}), and $W_{-1}=Q^i_{-2}=0$.
\end{thm}

\begin{proof}
Comparing the parts involving $dt$ of formulas (\ref{dPsi2p1}) and
(\ref{dPsi2p-Lem}) and substituting (\ref{dPsi2p-1st term},
\ref{dPsi2p-2nd term}) into (\ref{dPsi2p-Lem}), we obtain
\begin{equation}\label{formula-in pf thm}
\frac{1}{n!}\frac{\partial}{\partial
t}\Theta_{2p}\Big|_{t=0}=\frac{1}{n!}d_M\Phi_{2p}
+\langle-(n-2p)H^{f}_{2p+1}+pW_{2p-1},\nu\rangle
dV_{M}+p\sum\limits_{i,\alpha} a_{i\alpha}q^{i,\alpha}_{2p-2}dV_{M},
\end{equation}
where $q^{i,\alpha}_{2p-2}$ is the coefficient of $e_\alpha$ in
$Q^i_{2p-2}$, \emph{i.e.}, $Q^i_{2p-2}:=\sum_\alpha
q^{i,\alpha}_{2p-2}e_\alpha$.

Recall that we have the following formula concerning the functions
$a_i,a_\alpha,a_{i\alpha}$ in (\ref{def-var-frame}) (cf.
\cite{Chern}):
\[
\sum_ia_{i\alpha}\theta_{i}=d_Ma_\alpha+\sum_{\beta}a_\beta\theta_{\beta\alpha}+\sum_{i}a_i\theta_{i\alpha}=Da_{\alpha},
\]
which implies immediately
\[
\sum\limits_{i,\alpha}a_{i\alpha}q^{i,\alpha}_{2p-2}=\sum\limits_{i}\langle
Q^i_{2p-2},\nabla_{e_i}\nu\rangle.
\]
Then taking use of (\ref{variation1}) and integrating
(\ref{formula-in pf thm}) over $M$, we complete the proof.
\end{proof}
\begin{rem}
Recalling the expression of $\Phi_{2p}$ in (\ref{def-Phi2p}), if we
assume that $M$ is closed or the variation satisfies
$a_i(x)=0,\:a_{i\alpha}(x)=0$ for $x\in
\partial M$, the first variational formula turns to:
\begin{equation}\label{1st var-for}
\frac{d}{dt}\mathcal {M}_{2p}(f_t)\Big|_{t=0}=\int_M\Big(\langle
-(n-2p)H^f_{2p+1}+pW_{2p-1},\nu\rangle+p\sum\limits_i\langle
Q^i_{2p-2},\nabla_{e_i}\nu\rangle\Big)dV_M.
\end{equation}
\end{rem}

\begin{thm}\label{thm-ELeqn}
Let $f:M^n\rightarrow N^{n+m}$ be an isometric immersion from a
closed Riemannian manifold $M$ into a Riemannian manifold $N$. Then
for $p=0,1,\cdots,[\frac{n}{2}]$, the Euler-Lagrange equation for
the first variational formula of the total $2p$-th mean curvature
$\mathcal{M}_{2p}(f)$ is given by:
\begin{equation*}
L_{2p}:=-(n-2p)H^f_{2p+1}+pW_{2p-1}+p\widetilde{Q}_{2p-2}=0.
\end{equation*}
Here $H^f_{2p+1}$, $W_{2p-1}$ are the same with those in Theorem
\ref{Thm-varformula}, and
$$\widetilde{Q}_{2p-2}=\sum_{i,A}\langle Q^i_{2p-2},\nabla_{e_i}e_A\rangle
e_A-\sum_\alpha div (\sum_i q^{i,\alpha}_{2p-2}e_i)e_\alpha,$$ where
$Q^i_{2p-2}=\sum_\alpha q^{i,\alpha}_{2p-2}e_\alpha$ is defined in
Theorem \ref{Thm-varformula} and denote $\widetilde{Q}_{-2}=0$.
Henceforth, we call $M$ relatively 2p-minimal if $L_{2p}=0$.
\end{thm}
\begin{proof}
It suffices to treat with the term involving covariant derivative of
$\nu$ in (\ref{1st var-for}). Recall that $\nu=\sum_A a_A e_A$ and
$Q^i_{2p-2}=\sum_\alpha q^{i,\alpha}_{2p-2}e_\alpha$. Then
\begin{eqnarray*}
\sum_i\langle Q^i_{2p-2},\nabla_{e_i}\nu\rangle
&=&\sum_i\Big\langle Q^i_{2p-2}, \sum_A e_i(a_A)e_{A}+\sum_Aa_A\nabla_{e_i} e_A\Big\rangle\\
&=&\sum_{i,\alpha} q^{i,\alpha}_{2p-2}e_i(a_\alpha)+\sum_{i,A}a_A\langle Q^i_{2p-2},\nabla_{e_i}e_A\rangle \\
&=&\sum_\alpha div(\sum_ia_\alpha
q^{i,\alpha}_{2p-2}e_i)-\Big\langle \sum_\alpha
div(\sum_iq^{i,\alpha}_{2p-2}e_i)e_\alpha,\nu\Big\rangle+\Big\langle\sum_{i,A}
\langle Q^i_{2p-2},\nabla_{e_i}e_A\rangle e_A,\nu\Big\rangle\\
&=&\sum_\alpha div(\sum_ia_\alpha q^{i,\alpha}_{2p-2}e_i)+\langle
\widetilde{Q}_{2p-2},\nu\rangle.
\end{eqnarray*}
Thus according to Stokes' theorem, one can easily find that
\begin{equation*}
\frac{d}{dt}\mathcal {M}_{2p}(f_t)\Big|_{t=0}=\int_M \sum_\alpha
div(\sum_ia_\alpha q^{i,\alpha}_{2p-2}e_i)dV_M+\int_M\langle
L_{2p},\nu\rangle dV_M=\int_M\langle L_{2p},\nu\rangle dV_M,
\end{equation*}
which completes the proof of the theorem.
\end{proof}

When $N$ is a real space form $\mathbb{R}^{n+m}(c)$ with constant
sectional curvature $c$, one can find that
\begin{equation}\label{ELeqn-real-sp-form}
L_{2p}=-(n-2p)H^f_{2p+1}+2cpH^f_{2p-1},
\end{equation}
which was proved by \cite{Li1} firstly with different notations.

\section{Closed complex submanifolds in $\mathbb{C}P^{n+m}$}
In this section we prove that closed complex submanifolds in complex
projective spaces are relatively $2p$-minimal, \emph{i.e.}, critical
for the functional $\mathcal {M}_{2p}$ for all $p$.

Let $N$ be the complex projective space $\mathbb{C}P^{n+m}(c)$ with
constant holomorphic sectional curvature $c$. Denote by $J$,
$\langle,\rangle$ the almost complex structure and Hermitian metric
respectively. It is well known that the curvature tensor of $N$ can
be written as
\begin{eqnarray*}
R(X,Y,Z,T)&=&\frac{c}{4}\Big(\langle X,Z \rangle\langle Y,T
\rangle-\langle Y,Z \rangle\langle X,T
\rangle\\
&&+\langle JX,Z \rangle\langle JY,T \rangle-\langle JY,Z
\rangle\langle JX,T \rangle+2\langle JX,Y \rangle\langle JZ,T
\rangle\Big).
\end{eqnarray*}
Suppose $M$ is a complex submanifold of complex dimension $n$ in
$N$. Around each point $x$ in $M$, we can choose a local orthonormal
frame $\{e_1,\ldots,e_{2n+2m}\}$ of $TN$ such that
$e_{2}=Je_{1},\ldots,e_{2n+2m}=Je_{2n+2m-1}$, and
$e_1,\ldots,e_{2n}$ are tangent to $M$. In this section, we still
use $i,j,k$ (resp. $\alpha,\beta,\gamma$), \emph{etc.} for the
indices of tangent (resp. normal) vectors of $M$. In addition, for
simplicity we will use the following notations
\begin{equation*}
 \bar e_i:=e_{\bar i}:=Je_i,\quad \bar
e_\alpha:=e_{\bar \alpha}:=Je_\alpha.
\end{equation*}
Under this setting we can write the curvature tensor of $N$ over $M$
in a simpler form. For example,
\begin{equation}\label{curv-CPnm}
R_{i\alpha jk}=0,\quad R_{i\alpha
j\beta}=\frac{c}{4}(\delta^{i}_{j}\delta^{\alpha}_{\beta}+\delta^{\bar
i}_{j}\delta^{\bar \alpha}_{\beta}).
\end{equation}

The following Lemmas will be useful in the proof of Theorem
\ref{thm-CPn} later.
\begin{lem}\label{Lem-2ndff}
With the same notations as above, we get the following identity
about the second fundamental form of $M$:
\begin{equation}
\theta_{i\alpha}(e_{j})=\theta_{i\bar\alpha}(\bar
e_{j})=-\theta_{\bar i\alpha}(\bar e_{j}).
\end{equation}
\end{lem}
\begin{proof}
Straightforward computation shows
\begin{equation*}
\theta_{i\bar\alpha}(\bar e_j)=\theta_{\bar j\bar\alpha}(e_i)
=\langle\nabla_{e_{i}}\bar e_j, \bar e_{\alpha}\rangle=\langle
J\nabla_{e_{i}}e_{j}, J e_\alpha\rangle=\langle\nabla_{e_{i}}e_{j},
e_\alpha\rangle=\theta_{i\alpha}(e_{j}).
\end{equation*}
Similarly,
\begin{equation*}
\theta_{\bar i\alpha}(\bar e_{j})=\langle \nabla_{\bar e_j}\bar
e_i,e_\alpha\rangle =-\langle\nabla_{\bar e_j}e_i,\bar
e_\alpha\rangle=-\langle\nabla_{e_i}\bar e_j,\bar e_\alpha\rangle
=-\langle\nabla_{e_i}e_j,e_\alpha\rangle=-\theta_{i\alpha}(e_j).
\end{equation*}
\end{proof}
\begin{lem}\label{Lem-rel-cur}
With the same notations as above, we get the following identity:
\begin{equation*}
\sum_s\Omega_{I_{2p}}(X_1,\dots,JX_s,\dots,X_{2p})=0,
\end{equation*}
where $\Omega_{I_{2p}}$ is defined in (\ref{def-Omega_I2p}),
$X_{1},\ldots,X_{2p}$ are $2p$ vectors tangent to $M$.
\end{lem}
\begin{proof}
Obviously $M$ is also a K\"{a}hler manifold. Thus the formula
$\Omega_{ij}(JX_1,JX_2)=\Omega_{ij}(X_1,X_2)$ holds. We prove this
Lemma by induction. For $p=1$, it is not difficult to see that
$$\Omega_{i_1i_2}(JX_1,X_2)+\Omega_{i_1i_2}(X_1,JX_2)=0.$$ Suppose
the identity holds for $p-1$, then for $p$,
\begin{eqnarray*}
&&\sum_s\Omega_{I_{2p}}(X_1,\dots,JX_s,\dots,X_{2p})\\
&=&\sum_{t<s}(-1)^{s+t-1}\Omega_{I_{2p-2}}(X_1,\dots,\hat X_t,\dots
,\hat
X_s,\dots,X_{2p})\Omega_{i_{2p-1}i_{2p}}(X_t,JX_s)\\
&&+\sum_{t>s}(-1)^{s+t-1}\Omega_{I_{2p-2}}(X_1,\dots,\hat X_s,\dots
,\hat X_t,\dots,X_{2p})\Omega_{i_{2p-1}i_{2p}}(JX_s,X_t)\\
&&+\sum_{s}\sum_{t_1,t_2\neq
s,t_1<t_2}(-1)^{t_1+t_2-1}\Omega_{I_{2p-2}}(X_1,\dots,\hat
X_{t_1},\dots
,\hat X_{t_2},\dots,JX_s,\dots,X_{2p})\Omega_{i_{2p-1}i_{2p}}(X_{t_1},X_{t_2})\\
&=&\sum_{t<s}(-1)^{s+t-1}\Omega_{I_{2p-2}}(X_1,\dots,\hat X_t,\dots
,\hat X_s,\dots,X_{2p})\Big(\Omega_{i_{2p-1}i_{2p}}(X_t,JX_s)+\Omega_{i_{2p-1}i_{2p}}(JX_t,X_s)\Big)\\
&&+\sum_{t_1<t_2}\sum_{s\neq
t_1,t_2}(-1)^{t_1+t_2-1}\Omega_{I_{2p-2}}(X_1,\dots,\hat
X_{t_1},\dots,\hat X_{t_2},\dots,JX_s,\dots,X_{2p})\Omega_{i_{2p-1}i_{2p}}(X_{t_1},X_{t_2})\\
&=&\sum_{t_1<t_2}\sum_{s\neq
t_1,t_2}(-1)^{t_1+t_2-1}\Omega_{I_{2p-2}}(X_1,\dots,\hat
X_{t_1},\dots,\hat
X_{t_2},\dots,JX_s,\dots,X_{2p})\Omega_{i_{2p-1}i_{2p}}(X_{t_1},X_{t_2}).
\end{eqnarray*}
By assumption, the sum $\sum_{s\neq
t_1,t_2}\Omega_{I_{2p-2}}(X_1,\dots,\hat X_{t_1},\dots,\hat
X_{t_2},\dots,JX_s,\dots,X_{2p})$ equals zero for all $t_1,t_2$. In
conclusion, the proof is complete.
\end{proof}

\begin{thm}\label{thm-CPn}
Let $M$ be a closed complex submanifold of complex dimension $n$ in
$\mathbb{C}P^{n+m}$, then
\begin{equation*}
L_{2p}=-(2n-2p)H^f_{2p+1}+\frac{cp(2n-2p)}{2(2n-2p+1)}H^f_{2p-1}=0,
\end{equation*}
\emph{i.e.}, $M$ is relatively $2p$-minimal for $p=0,1,\ldots,n$.

\end{thm}
\begin{proof}
Clearly $\widetilde{Q}_{2p-2}=0$ since now $R_{i\alpha jk}=0$ by
(\ref{curv-CPnm}). Therefore to calculate $L_{2p}$ in Theorem
\ref{thm-ELeqn}, it suffices to compute $W_{2p-1}$. Combining the
definition of $H^f_{2p-1}$ and Lemma \ref{Lem-2ndff}, we compute it
as follows:
\begin{eqnarray*}
W_{2p-1}&=&\frac{2(2n-2p)!}{(2p-1)!(2n)!}\sum\limits_{\alpha,\beta}\sum\limits_{I_{2p-1},i}\sum\limits_{J_{2p-1},j}\delta ^{i_1,\ldots,i_{2p-1},i}_{j_1,\ldots,j_{2p-1},j}R_{{i}\beta j\alpha}\Omega_{I_{2p-2}}\wedge\theta_{i_{2p-1}\beta}(e_{j_1},\ldots,e_{j_{2p-1}})e_\alpha\\
&=&\frac{c(2n-2p+1)!}{2(2p-1)!(2n)!}\sum\limits_{\alpha}\sum\limits_{I_{2p-1}}\sum\limits_{J_{2p-1}}\delta^{i_1,\ldots,i_{2p-1}}_{j_1,\ldots,j_{2p-1}}\Omega_{I_{2p-2}}\wedge\theta_{i_{2p-1}\alpha}(e_{j_1},\ldots,e_{j_{2p-1}})e_\alpha\\
&&+\frac{c(2n-2p)!}{2(2p-1)!(2n)!}\sum\limits_{\alpha}\sum\limits_{I_{2p-1},i}\sum\limits_{J_{2p-1}}\delta^{i_1,\ldots,i_{2p-1},i}_{j_1,\ldots,j_{2p-1},\bar{i}}\Omega_{I_{2p-2}}\wedge\theta_{i_{2p-1}\bar\alpha}(e_{j_1},\ldots,e_{j_{2p-1}})e_\alpha\\
&=&\frac{c}{2}H^f_{2p-1}-\frac{c(2n-2p)!}{2(2n)!}\sum\limits_{\alpha}\sum\limits_{I_{2p-1}}\sum\limits_{s=1}^{2p-1}\Omega_{I_{2p-2}}\wedge\theta_{i_{2p-1}\bar\alpha}(e_{i_1},\ldots,\bar e_{i_s},\ldots,e_{i_{2p-1}})e_\alpha\\
&\triangleq&\frac{c}{2}H^f_{2p-1}-\frac{c(2n-2p)!}{2(2n)!}\sum\limits_{\alpha}\sum\limits_{I_{2p-1}}\sum\limits_{s=1}^{2p-1}(-1)^{s-1}\Omega_{I_{2p-2}}(e_{i_1},\ldots,\widehat{e}_{i_s},\ldots,e_{i_{2p-1}})\theta_{i_{2p-1}\bar\alpha}(\bar e_{i_s})e_\alpha\\
&=&\frac{c}{2}H^f_{2p-1}-\frac{c(2n-2p)!}{2(2n)!}\sum\limits_{\alpha}\sum\limits_{I_{2p-1}}\sum\limits_{s=1}^{2p-1}(-1)^{s-1}\Omega_{I_{2p-2}}(e_{i_1},\ldots,\widehat{e}_{i_s},\ldots,e_{i_{2p-1}})\theta_{i_{2p-1}\alpha}( e_{i_s})e_\alpha\\
&=&\frac{c}{2}H^f_{2p-1}-\frac{c(2n-2p)!}{2(2n)!}\sum\limits_{\alpha}\sum\limits_{I_{2p-1}}\Omega_{I_{2p-2}}\wedge\theta_{i_{2p-1}\alpha}(e_{i_1},\ldots,e_{i_{2p-1}})e_\alpha\\
&=&\frac{c(2n-2p)}{2(2n-2p+1)}H^f_{2p-1},
\end{eqnarray*}
where $``\triangleq"$ is deduced by Lemma \ref{Lem-rel-cur}.
Therefore, we obtain
\begin{equation*}
L_{2p}=-(2n-2p)H^f_{2p+1}+\frac{cp(2n-2p)}{2(2n-2p+1)}H^f_{2p-1}.
\end{equation*}

Meanwhile, a direct calculation shows that $H^f_{2p+1}$ of $M$
vanishes for each $p$. In fact, combining the fact that
$\Omega_{\bar i\bar j}(\bar e_{k},\bar e_{l})=\Omega_{ij}(e_{k},
e_{l})$ with Lemma \ref{Lem-2ndff}, we get
\begin{eqnarray*}
H^f_{2p+1}&=&\frac{(n-2p-1)!}{n!}\sum\limits_{\alpha}\sum\limits_{I_{2p+1}}\Omega_{I_{2p}}\wedge\theta_{i_{2p+1}\alpha}(e_{i_1},\ldots,e_{i_{2p+1}})e_\alpha\\
&=&\frac{(n-2p-1)!}{n!}\sum\limits_{\alpha}\sum\limits_{I_{2p+1}}\Omega_{\bar i_1\bar i_2}\wedge\dots\wedge\Omega_{\bar i_{2p-1}\bar i_{2p}}\wedge\theta_{\bar i_{2p+1}\alpha}(\bar e_{i_1},\ldots,\bar e_{i_{2p+1}})e_\alpha\\
&=&\frac{(n-2p-1)!}{n!}\sum\limits_{\alpha,s}\sum_{I_{2p+1}}(-1)^{s-1}\Omega_{\bar i_1\bar i_2}\wedge\dots\wedge\Omega_{\bar i_{2p-1}\bar i_{2p}}(\bar e_{i_1},\ldots,\widehat e_{i_s},\ldots,\bar e_{i_{2p+1}})\theta_{\bar i_{2p+1}\alpha}(\bar e_{i_s})e_\alpha\\
&=&-\frac{(n-2p-1)!}{n!}\sum\limits_{\alpha,s}\sum_{I_{2p+1}}(-1)^{s-1}\Omega_{i_1i_2}\wedge\dots\wedge\Omega_{i_{2p-1}i_{2p}}(e_{i_1},\ldots,\widehat e_{i_s},\ldots,e_{i_{2p+1}})\theta_{ i_{2p+1}\alpha}(e_{i_s})e_\alpha\\
&=&-\frac{(n-2p-1)!}{n!}\sum\limits_{\alpha}\sum\limits_{I_{2p+1}}\Omega_{I_{2p}}\wedge\theta_{i_{2p+1}\alpha}(e_{i_1},\ldots,e_{i_{2p+1}})e_\alpha\\
&=&-H^f_{2p+1}=0.
\end{eqnarray*}
This completes the proof of the theorem.
\end{proof}

\section{Relatively $2p$-minimal and austere submanifolds}
In this section, we discuss the relations between relatively
$2p$-minimal submanifolds and austere submanifolds in real space
forms, as well as a special variational problem.

Let $f:M^n\rightarrow \mathbb{R}^{n+m}(c)$ be an isometric immersion
in a real space form of constant sectional curvature $c$. Recall
that the volume of any tubular hypersurface $M^f(r)$ with radius $r$
($0<r<\varepsilon$) of $M^n$ in $\mathbb{R}^{n+m}(c)$ is given by
the well known Weyl-Gray tube formula (cf. \cite{Ge})
\begin{equation}\label{W-G-tube}
V(M^f(r))=\sum_{p=0}^{[\frac{n}{2}]}\frac{C_{m+2p-1}}{2^{2p}\pi^{p}p!}(^{\hskip
0.05cm n}_{2p})(2p)!\mathcal {M}_{2p}(f)
(\cos(r\sqrt{c}))^{n-2p}(\frac{\sin(r\sqrt{c})}{\sqrt{c}})^{m+2p-1},
\end{equation}
where $C_{m+2p-1}$ is the volume of $S^{m+2p-1}(1)$, the $sin$,
$cos$ functions are considered as complex functions, and $\mathcal
{M}_{2p}(f)$ is the total $2p$-th mean curvature of $f$. Put
$\mathcal {V}_r(f):=V(M^f(r))$. Then $\{\mathcal {V}_r\mid
0<r<\varepsilon\}$ forms a one-parameter family of functionals over
isometric submanifolds in $\mathbb{R}^{n+m}(c)$. We call $M$ a
\emph{tubular minimal} submanifold of $\mathbb{R}^{n+m}(c)$ if it is
a critical point of $\mathcal {V}_r$ for all $0<r<\varepsilon$.
Observing the Weyl-Gray tube formula (\ref{W-G-tube}), we find that
$M$ is a critical point of $\mathcal {V}_r$ for all
$0<r<\varepsilon$ if and only if it is a critical point of $\mathcal
{M}_{2p}$ for all $p=0,1,\ldots,[\frac{n}{2}]$, or equivalently, it
is $2p$-minimal for all $p=0,1,\ldots,[\frac{n}{2}]$. Combining
these with the Euler-Lagrange equation (\ref{ELeqn-real-sp-form})
and the second identity in (\ref{rel-Hf-HM}), we deduce the
following
\begin{prop}\label{equiv-tubularmin}
Let $f:M^n\rightarrow \mathbb{R}^{n+m}(c)$ be an isometric immersion
in a real space form of constant sectional curvature $c$. Then the
following are equivalent:
\begin{itemize}
\item[(i)] $M$ is tubular minimal;
\item[(ii)] $M$ is relatively $2p$-minimal, \emph{i.e.}, $L_{2p}=-(n-2p)H^f_{2p+1}+2cpH^f_{2p-1}=0$ for all
$p=0,1,\ldots,[\frac{n}{2}]$;
\item[(iii)] $H^f_{2p+1}=0$ for all
$p=0,1,\ldots,[\frac{n}{2}]$;
\item[(iv)] $M$ is $2p$-minimal, \emph{i.e.}, $H^M_{2p+1}=0$ for all $p=0,1,\ldots,[\frac{n}{2}]$.
\end{itemize}
\end{prop}

Recall that a submanifold of a Riemannian manifold is called
\emph{austere} by Harvey and Lawson \cite{HL} if its principle
curvatures in any normal direction occur in oppositely signed pairs.
They showed, among other fundamental results on calibrated geometry,
austere submanifolds of Euclidean space are exactly those whose
co-normal bundle is special lagrangian and hence absolutely
minimizing. Except for the case of surfaces, austerity is much
stronger than minimality. Many examples and (partially)
classifications of austere submanifolds of Euclidean space have been
established by several authors, such as \cite{Br}, \cite{DF},
\cite{IST}, \emph{etc}. For minimal $3$-folds in different space
forms, \cite{CL} gives a local classification of the submanifolds
for which the equality in the DDVV inequality (also called the normal scalar
curvature inequality which was proved independently and differently by \cite{Lu} and \cite{GT})
holds everywhere and hence austere. Note that by the pointwise equality condition for the DDVV
inequality given by \cite{GT} (also discussed in \cite{Lu}), minimality together with this DDVV
equality is sufficient for austerity. It is worthy to mention that the classification problem of submanifolds attaining the DDVV equality everywhere still remains a rather interesting open problem (see \cite{GT2} for a brief introduction and \cite{Lu}, \cite{DT2} for some advances). As far as we compare austerity
with tubular minimality, we derive the following
\begin{prop}
Let $M^n$ be an $n$-dimensional austere submanifold of the real
space form $\mathbb{R}^{n+m}(c)$. Then $M$ is tubular minimal.
Moreover, each $2p$-th mean curvature satisfies
$(-1)^pK^f_{2p}\geq0$.
\end{prop}
\begin{proof}
By the definition of austerity, we see that each odd order
elementary symmetric polynomial $M_{2p+1}(\xi)$ of the shape
operator $S_{\xi}$ with respect to any unit normal vector $\xi$ of
$M$ vanishes. Recalling that in \cite{Ge} it is proved that
\begin{equation*}
H^f_{2p+1}=\frac{2^{2p}\pi^pp!(m+2p)}{C_{m+2p-1}(2p+1)!}\int_{S^{m-1}(1)}\xi
M_{2p+1}(\xi)d\xi,
\end{equation*}
we get $H^f_{2p+1}=0$ for all $p=0,1,\ldots,[\frac{n}{2}]$, and
hence by Proposition \ref{equiv-tubularmin}, $M$ is tubular minimal.
Moreover, austerity implies that the $2p$-th elementary symmetric
polynomial $M_{2p}(\xi)$ of the shape operator $S_{\xi}$ has the
sign of $(-1)^p$, which then shows $(-1)^pK^f_{2p}\geq0$ by the
following integral formula (cf. \cite{Ge}):
\begin{equation*}
K^f_{2p}=\frac{2^{2p}\pi^pp!}{C_{m+2p-1}(2p)!}\int_{S^{m-1}}M_{2p}(\xi)
d\xi.
\end{equation*}
The proof is now complete.
\end{proof}

\begin{ack}
It is our great pleasure to thank Professor Zizhou Tang for his
guidance and support. Many thanks also to Professors Haizhong Li,
Jiagui Peng and Changping Wang for their helpful discussions and
useful suggestions during the preparation of this paper.
\end{ack}

\bibliographystyle{amsplain}

\end{document}